\documentclass[12pt]{amsart}
\usepackage{amsmath,amsthm,latexsym,amscd,amsbsy,amssymb,amsfonts,fleqn,leqno,
euscript, pb-diagram}

\let\SavedRightarrow=\Rightarrow
\usepackage{marvosym}
\let\Rightarrow=\SavedRightarrow

\newcommand{\Bee }{\mathcal B}

\newcommand{\Iee }{\mathcal I}
\newcommand{\Wee }{\mathcal W}

\newcommand{\Uee }{\mathcal U}

\renewcommand{\int}{\operatorname{Int}}

\newtheorem{thm}{Theorem}
\newtheorem{pro}[thm]{Proposition}
\newtheorem{lem}[thm]{Lemma}

\newtheorem{cor}[thm]{Corollary}

\title{Embeddable properties of metric $\sigma$-discrete  spaces}

\author{Szymon Plewik}
\address{Institute of Mathematics, University of Silesia, ul. Bankowa 14, 40-007 Katowice}
\email{plewik@math.us.edu.pl}

\author{Marta Walczy\'nska}
\address{Institute of Mathematics, University of Silesia, ul. Bankowa 14, 40-007 Katowice}
\email{mwalczynska@us.edu.pl}

\begin{document}

\maketitle

\begin{abstract} Dimensional types of metric scattered spaces are investigated.   Revised proofs of Mazurkiewicz-Sierpi\'nski and Knaster-Urbanik theorems  are presented. Embeddable properties of  countable metric  spaces are generalized onto uncountable metric $\sigma$-discrete  spaces.  Some related topics are also explored. For example: For each infinite cardinal number $\frak m$, there exist $2^{\frak m}$ many  non-homeomorphic metric scattered spaces of the cardinality $\frak m $; If $X \subseteq \omega_1$ is a stationary set, then the poset formed from  dimensional types of subspaces of $X$ contains uncountable  anti-chains and uncountable strictly descending chains. 
\end{abstract}

 \section{Introduction} \label{s1}

Suppose $X $ and $Y$ are topological   spaces. The symbol  $ X <_E Y$ means that $X$ is 
homeomorphic to a subspace of $Y$. If  $ X <_E Y$, then we say that $X$ has a dimensional type  smaller  or equal to  the dimensional type of $Y$. When $ X <_E Y$ and $ Y <_E X$, then $X$ and $Y$ have the same dimensional type, what we  denote briefly $X=_E Y$.  When  $ X <_E Y$ and is not fulfilled $ Y <_E X$, then  $X$ has a smaller dimensional type   than    $Y$.
  First time the relation $<_E$ was investigated by M. Fr$\acute{\mbox{e}}$chet \cite{fre}.  In \cite[p. 24]{sie1} W. Sierpi\'nski cites alternative names for dimensional types: type de dimensions,  Fr$\acute{\mbox{e}}$chet; Hom$\ddot{\mbox{o}}$ie, Mahlo. Basic properties and definitions relating to dimensional types are also discussed in textbooks  \cite{sie}, \cite{kur1} and \cite{kur}. K. Kuratowski uses the name topological rank for  dimensional type, \cite[p. 112]{kur1}.   
 It is widely known -  some authors treat them like a mathematical folklore, compare \cite{gil} - the following results. \par In \cite{ms} S. Mazurkiewicz and W. Sierpi\'nski proved the following two facts. \textit{There is continuum many  non-homeomorphic countable metric and scattered spaces}. \textit{A countable compact metric space $X$ is homeomorphic to the ordinal $\omega^\alpha n +1$}. In the second claim  $n=|X^{(\alpha)}| $ is a natural number and $X^{(\alpha)} $ is the first discrete derivative of $X$, where $\alpha \in \omega_1$. The countable ordinal $\omega^\alpha n +1$ is equipped with the order topology. \par    B. Knaster and K. Urbanik \cite{ku}:  \textit{Any countable metric scattered  space has a metric scattered compactification.}   An alternative proof is given in the monograph \cite[Theorem 6, p. 25]{kur1}.  \par  R. Telg$\acute{\mbox{a}}$rsky \cite[Theorem 9]{tel}: \textit{Any metric scattered space can be embedded into  a sufficiently large  ordinal number}. Independently, the same is also proved in \cite{af}.
 
The poset { $(P(\mathbb Q), <_E)$}, where   $P(\mathbb Q)$ is the family of all subsets of the rational numbers  $\mathbb Q$, is  described   by  W.D. Gillam in the paper \cite{gil}.  The  set ${\mathcal P}(\Bbb Q)/\!\!\!=_E$ of all equivalence classes $[X]=\{Y\subseteq \Bbb Q:Y=_E X\}$ is partially ordered by the relation $[ X]\leq_d [Y]$ whenever $X<_E Y.$ In \cite{gil}, it is shown that the poset $({\mathcal P}(\Bbb Q)/\!=_E ,\leq_d)$ has cardinality $\omega_1$ and  $[\Bbb Q]$  is the only element with $\omega_1$ many elements below it. Moreover, $({\mathcal P}(\Bbb Q)/\!\!\!=_E,\leq_d)$  lacks infinite anti-chains and infinite strictly descending chains. In fact,   $({\mathcal P}(\Bbb Q)/\!\!\!=_E,\leq_d)$ is described  using the  Cantor-Bendixson rank,   local homeomorphism invariants and  local embeddable properties regarding the position  of points in a countable metric scattered  space. Initially, we believed that
analogous invariants should work successfully in the case of  uncountable metric scattered    spaces. Now, we are going to check the rationality of those beliefs.

		For any space $X$, the $\alpha$-derivative of $X$, which is denoted  $X^{(\alpha)}$, is defined  inductively: 
			 $X^{(0)}=X$;
		 $X^{(\alpha+1)}= \{x\in X^{(\alpha)}: x \mbox{ is not isolated in } X^{(\alpha)}\};$
	 $X^{(\alpha)} = \bigcap\{X^{(\beta)}: \beta < \alpha\}$ for a limit ordinal $\alpha$.
			Thus, each $ X^{(\alpha)}$ is a closed subset of $X$. 	
	If there exists an ordinal $\alpha$ such that $X^{(\alpha)}= \emptyset$, then $X$ is  called a \textit{scattered}  space.  The smallest ordinal such that $X^{(\alpha)}=\emptyset$ is denoted $N(X)$ and is  called the Cantor-Bendixson  \textit{rank} of $X$. Other notions of set theory and topology will be used according to textbooks \cite{eng} and  \cite{kun}. In particular, the sum of  topological spaces  we use like in the  book \cite[p. 103]{eng}.
	
		The paper is organized as follows. The results, which we consider completely new ones are formulated as theorems or lemmas. Modifications of known facts or facts from  mathematical folklore are formulated as propositions or corollaries. Proofs of propositions refer to the original idea of S. Mazurkiewicz and W. Sierpi\'nski relying on the use of ordinal arithmetic. In fact, we   extend this arithmetic by  adding a new element, i.e. the subspace  $ I\subset \omega^2+1 $, compare Section \ref{s5}. 
		Our intention is to initiate research directions  of dimension types in terms of ordinals and  metric $\sigma$-discrete spaces. So, we carefully analyze the tools that have been used successfully in  countable cases.

\section{Remarks on ordinal arithmetic} \label{s2}
Ordinal arithmetic is comprehensively described in many textbooks of modern set theory, and so we only briefly discuss  aspects we need. Topological properties of subsets of ordinals will be considered only with the order topology, i.e. the topology generated by  open rays $ \{\beta: \beta <\alpha \} $ and $ \{\beta: \beta> \alpha \} $, where $ \alpha $ is an  ordinal. So,  we  reconsider schemes of ordinal arithmetic, which were used in the paper by S. Mazurkiewicz and W. Sierpi\'nski \cite{ms}.
For ordinal numbers, we will  use the convention $\alpha = \{\beta: \beta < \alpha\}$. If $ \beta \in \alpha $,  we write $\beta < \alpha $, except for phrases $n\in \omega$, where $n$ is a finite ordinal and $\omega$ is the first infinite ordinal. 	Suppose $\alpha$ and $\beta$ are ordinals, then $ \alpha + \beta$ is the unique  ordinal $\gamma$ which is isomorphic to  a copy of $\alpha$ followed by a copy of $\beta$. The  addition of ordinals is  associative, but not  commutative. Also $\beta < \alpha$ implies $\beta +\gamma \leq \alpha + \gamma$, for any ordinal $\gamma$.
	The  ordinal $\gamma$	added $n$-times  is denoted $\gamma \cdot n$.  If $\{ \lambda_n: n \in  \omega \}$ is a sequence of ordinals, then  $$\sum_{n \in \omega} \lambda_n = \sup \{\lambda_0 + \lambda_1 + \ldots + \lambda_n: n \in \omega\}.$$	
	The following limit ordinals are important because of  the above mentioned    Mazurkiewicz-Sierpi\'nski theorem. Put $ \omega^0=1 $, $\omega^1 = \omega$  and  define the countable limit ordinal $$\omega^\alpha = \sup \{\omega^{\beta}\cdot n: \beta <\alpha \mbox{ and } 0<n\in \omega\},$$  for each countable ordinal $\alpha$.  			
			 If $\beta < \omega^\alpha$, then the interval $(\beta,  \omega^\alpha)$ is isomorphic to $\omega^\alpha = [\emptyset, \omega^\alpha)$ and also these intervals  are homeomorphic.    
	 If  $\beta < \omega^\alpha$, then $$\omega^\alpha +1 = \beta + \omega^\alpha +1  =_E  \omega^\alpha + \beta +1.$$	
	 If $\gamma > \sup \gamma$  is  a
		countable  infinite ordinal, then there exist $ n\in \omega$ and an ordinal $\alpha$ such that   	$\omega^\alpha \cdot n +1 =_E \gamma$.
 If $\gamma$ is a limit ordinal, then there exist $n\in \omega$ and  ordinals $\alpha$ and $\beta$ such that the subspace $\omega^\alpha \cdot n +1 \setminus \{\beta \} \subseteq  \omega^\alpha \cdot n +1$ is homeomorphic to $ \gamma$.	
We omit  details of mentioned above facts. Instead of this, we  present  the following.

\textbf{Proposition.} \textit{If \, $0<\alpha$, then $N(\omega^\alpha)= \alpha$ and $N(\omega^\alpha+1)= \alpha+1$.} 
		\begin{proof} If $\alpha =1$, then $\omega +1$  is homeomorphic to a convergence sequence.  So, $(\omega +1)^{(1)}= \{\omega\}$ and $(\omega )^{(1)}= \emptyset$, hence  $N(\omega)=1$ and $N(\omega+1)=2$.
	
	Suppose, that the thesis holds for all non-zero $\beta < \alpha$. 
		If $\alpha = \beta +1$, using the induction assumptions, we get $$(\omega^\alpha +1)^{(\beta)}=\{  \omega^\beta \cdot n: 0<n \in \omega \} \cup \{\omega^\alpha\}=_E \omega +1.$$ Therefore 
	$(\omega^\alpha +1)^{(\alpha)}=\{\omega^\alpha \} $ and   $(\omega^\alpha )^{(\alpha)}=\emptyset.$ Hence $N(\omega^\alpha)=\alpha$ and  $N(\omega^\alpha+1)=\alpha + 1$.
	
	Suppose $\omega^\alpha = \sum_{n\in\omega} \omega^{\beta_n}$, where $\alpha = \sup_{n\in \omega}{\beta_n}$ is a limit ordinal. For any $\beta < \alpha$, by the induction assumptions, we have $$(\omega^\beta)^{(\alpha)}=\emptyset \mbox{\; and \;} \omega^\alpha  \in (\omega^\alpha +1)^{(\beta)} .$$  Bearing this in mind, we check that  $$(\omega^\alpha )^{(\alpha)}= \bigcup\{ ( \omega^{\beta_n} )^{(\alpha)}: n\in \omega \mbox{\; and \; } \beta_n < \alpha \}=\emptyset.$$  We still have  $\omega^\alpha  \in (\omega^\alpha +1)^{(\beta_n)},$  therefore  
	$(\omega^\alpha +1)^{(\alpha)}=\{\omega^\alpha \} .$  
	\end{proof}

  \section{On $\sigma$-discrete metric spaces} \label{s3}
	 A metric space is called \textit{$\sigma$-discrete}, if it is an union of countably many discrete subspaces. Any countable metric space, being countable sum of single points, is $\sigma$-discrete. In particular, 
  the space $\mathbb Q$ of all rational numbers is $\sigma$-discrete. 
	\begin{lem}\label{ll}
	Each  metric $\sigma$-discrete space $X$ is an union of countably many closed and discrete subspaces. \end{lem}
	
	\begin{proof} Use the Bing theorem \cite[4.4.8]{eng} in the following way. Let $$\mathcal B= \bigcup \{\mathcal B_n: n\in \omega \}$$ be a $\sigma$-discrete base for $X$,  where each $\mathcal B_n$ is a discrete family. And let $X_0, X_1, \ldots $ be discrete subspaces summing $X$. If $x\in X_k$, then 
	fix $V_x^m\in\mathcal  B_m$ such that  $V_x ^m\cap X_k = \{ x\}$. If there is  no relevant  $V_x^m$, then put $V_x^m=\emptyset$. And put $$X_{k,m} = X_k \cap \bigcup \{V_x^m: x\in X_k \mbox{ and }  V_x ^m\in\mathcal B_m \}$$ and then check that sets $X_{k,m}$ are such that we need. 	\end{proof}

	Let $B(\frak m) = \frak m^\omega $ be the Baire space of weight $\frak m$, where $\frak m$ is an infinite  cardinal. Since $0\in \frak m$, we can  put
	$$C(\frak m) = \{y\in B(\frak m): \mbox{almost all coordinates of $y$ are equal to 0} \}$$ and consider $C(\frak m)$
with the topology inherited from $B(\frak m)$. Each 
Baire space $B(\frak m)$  is  metric and each   $ C(\frak m)$ is a $\sigma$-discrete metric subspace. 
Note that $ C(\omega_0)$ is a homeomorphic copy of the rational numbers and the Baire space $B(\omega_0)$ is homeomorphic to the irrational numbers. Therefore is why the next proposition says that spaces $C(\frak m)$ are analogues of the rational numbers. A characterization of the rational numbers generalized by the next proposition  is usually attributed to G. Cantor,  L. E. J. Brouwer or W. Sierpi\'nski.
	
\textbf{Proposition.}	 \textit{A nonempty metric $\sigma$-discrete  space, with all nonempty open subsets of weight $\frak m$,  is homeomorphic to   $C(\frak m)$. 
	 A metric $\sigma$-discrete  space of the weight $\frak m$ is homeomorphic to a subspace of $C(\frak m)$.}
	\begin{proof} See T. Przymusi\'nski \cite{prz}, compare Sz. Plewik \cite{ple}. \end{proof}

	\textbf{Proposition.}
\textit{
	A metric $\sigma$-discrete   space contains a homeomorphic copy of the rational numbers or it is scattered. }	
		
	\begin{proof} Let $X$ be  a   metric $\sigma$-discrete space which is not  scattered. Thus $X$ contains a dense in itself subspace which,  being metric and dense in itself, has to contain a homeomorphic copy of the rational numbers.\end{proof}

\begin{thm}\label{ms} 
	Any metric scattered space  is $\sigma$-discrete. 	\end{thm}
		\begin{proof} 
K. P. Hart offered us the following elementary reasoning. Let  $(X,\varrho)$ be a metric scattered space. For every $x\in X$, let   $\alpha_x$ be the ordinal such that $ x\in X^{(\alpha_x)}$  and $ x\notin X^{(\alpha_x+1)}$ , and then fix a natural number $n_x$ such that $B(x,\frac{1}{n_x})\cap X^{(\alpha_x)}=\{x\}.$  Finally put $$D_n=\{x\in X: n_x=n\}.$$ If $x, y \in D_n$ and  $x \not= y$, then $\varrho (x,y) \geq \frac{1}{n}$. So, $X$ is the countable union of closed and discrete sets $D_n$. 
	\end{proof}

Applying metrization theorems -- for example the  Stone theorem, compare \cite[4.4.1]{eng} -- one obtains the following.
A metric locally  $\sigma$-discrete space is $\sigma$-discrete. And then one can check that if a metric space $X$ is not $\sigma$-discrete, then the set $$ \{x\in X: \mbox{ no neighborhood of $x$ is $\sigma$-discrete} \}$$ is dense in itself. It  gives us an other  proof of Theorem \ref{ms}.

Each metric $\sigma$-discrete  space  is paracompact in a stronger sense.

\begin{thm}\label{eoc}
Every  open cover $\mathcal{U}$ of  a  metric $\sigma$-discrete  space $X$ has a disjoint open refinement.    \end{thm}
\begin{proof} Modifying   Engelking's reasoning 1.3.2 from \cite{enw}, one can obtain the following. 
	If a normal space is an union of countably many closed and discrete subspaces, then it has a base consisting of closed-open sets. So, any metric $\sigma$-discrete  space has a base consisting of closed-open sets. 
	
	Let closed and discrete sets $X_{k,m} \subseteq X$ are defined as in the proof of lemma \ref{ll}.  Fix $k$ and $m$. The family $$ \{V^m_x: x \in X_{k,m}\} \subseteq\mathcal B_m$$ is discrete. So, we can choose the closed-open sets  $W^m_x \subseteq W \in \mathcal U$ and  $W^m_x \subseteq V^m_x$, for each $x\in X_{k,m}$,
	 such that $$X_{k,m} \subseteq \bigcup \{W^m_x: x \in X_{k,m}\} $$  and the union $\bigcup \{W^m_x: x \in X_{k,m}\}$ is closed-open. 	Note that, the family $\{W^m_x: x \in X_{k,m}\}$, being discrete,  consists of pairwise disjoint sets.
 Sets  $X_{k,m}$   enumerate as  $\{ Y_n: n \in \omega\} $. 
Let $\mathcal W_0= \{W^m_x: x \in  Y_0 =X_{k,m}\}.$ If  $ X_{k,m}=Y_n $ and  families of closed-open sets   $\mathcal W_0 , \mathcal W_1, \ldots,  \mathcal W_{n-1}$ are already defined such that unions $\cup \mathcal W_0 , \cup\mathcal W_1, \ldots, \cup \mathcal W_{n-1}$ are closed-open sets, then let $\mathcal W_n$ be the family $$ \{W^m_x\setminus \bigcup \{ \cup\mathcal W_i: i<n\}: x \in Y_n \mbox{ and } x \notin \bigcup \{ \cup\mathcal W_i: i<n\}\}.$$ The union $\bigcup\{\mathcal W_n: n \in \omega\}$ is a needed refinement of $\mathcal{U}$. \end{proof}

We get a modification of the  Telg$\acute{\mbox{a}}$rsky result, see \cite[Theorem 3]{tel}.
\begin{thm}\label{tel} 
Every base of a  metric $\sigma$-discrete  space  contains a locally finite open refinement. \end{thm}
\begin{proof} Let $X$ be a  metric $\sigma$-discrete  space such that $X= X_0 \cup X_1\cup \ldots ,$ where subspaces $X_n$ are  closed,  discrete and pairwise disjoint. Fix a base $\Bee$. Afterward, apply the following algorithm.  Choose a cover $\Uee \subseteq \Bee$ such that if $x\in X_0$ and $x\in A\in \Uee$, then  $  A \cap X_0= \{x\}$. By Theorem \ref{eoc}, the cover $\Uee$ has a disjoint open refinement $\Wee$. 
Choose  a refinement  $\Uee^* \subseteq \Bee$ and   a disjoint open refinement  $\Wee^*$  such that  $$\Wee^* \prec \Uee^* \prec \Wee \prec \Uee.$$ Without loss of generality, we can assume that there exists an unique neighborhood $A_x \in \Uee^*$ such that $A_x \cap X_0= \{x\}$, for each $x\in X_0$. Let $\mathcal U^0$ be the family of all such selected sets $A_x$. Thus, $\mathcal U^0$ and  $$\mathcal V^0= \{A \in \Wee^*: A \cap X_0 \not= \emptyset \}$$ are  discrete families.  Since  $\mathcal V^0 \subseteq \Wee^*$, then the union of all elements of $\mathcal V^0$ is a closed-open set.

 Assume that  discrete families $\mathcal U^0, \mathcal U^1, \ldots , \mathcal U^n$ and $\mathcal V^0, \mathcal V^1,
 \ldots ,\mathcal V^n$ are already defined such that the union $Y = \bigcup\{ A \in  \mathcal V^k: 0\leq k \leq n\}$ -- of closed-open and pairwise disjoint sets -- is a closed-open subset of $X$. Repeat the above algorithm by substituting $ X \setminus Y $ for  $ X $ and $\{A \in \Bee: A \subseteq X \setminus Y\}$ for $\Bee$ and $X_{n+1} \setminus Y$ for  $X_0$.
As a result, we get a discrete family $\mathcal U^{n+1} \subseteq \Bee$ and a discrete family $\mathcal V^{n+1}$ consisting of pairwise disjoint closed-open sets. From the properties of our algorithm  we get that $\mathcal U =\mathcal U^0 \cup\mathcal U^1\cup \ldots \subseteq  \Bee$ is a locally finite cover of $X$. Indeed, the family $\mathcal V^0 \cup\mathcal V^1\cup \ldots$ is a disjoint open refinement of $\mathcal U$. Also, each $A \in \mathcal V^n$ meets no  element of $\mathcal U^k$, for $k> n$. Thus, if $x\in A \in\mathcal V^n$, then  there exist open neighborhoods $B_0, B_1, \ldots , B_{n-1}$ of $x$ such that any $B_k$ meets at most  one element of the discrete family  $\mathcal U^k$. Therefore, the intersection $A \cap B_0 \cap B_1 \cap \ldots \cap  B_{n-1}$  meets at most finitely many elements of the cover $\mathcal U \subseteq \Bee$.  
 \end{proof}

In other words, Theorem \ref{tel} says that each  metric $\sigma$-discrete  space is totally paracompact. R.  Telg$\acute{\mbox{a}}$rsky  \cite{tel} only shows that metric scattered spaces are totally paracompact,  so  we receive a little stronger result.
Simplified versions of Theorem \ref{eoc} are applied in papers by V. Kannan and  M. Rajagopalan \cite[(1974)]{kr}, A. Arosio and A.V. Ferreira \cite[(1980)]{af} and R. Telg$\acute{\mbox{a}}$rsky \cite[(1968)]{tel}.  
		 Note that, for similar facts it is applied phrase "Every  { finite} open cover of  ..." in textbooks on dimension theory, 
		for example in \cite{enw}  or \cite{eng}. 		
		
		Now,   discuss  constructions which will be used in futher proofs. Let $\{X_\beta: \beta < \alpha \}$ be a family of scattered spaces such that $X_\beta^{(\beta)}= \{g_\beta\}.$ If additionally $\alpha$ is a limit ordinal, then let $J (\{X_\beta: \beta < \alpha \})	$ be the hedgehog space with spininess $X_\beta$. 	The hedgehog space is formed by gluing points $g_\beta$ into the point $g.$ The metric is determined such that  points of  $ X_\beta$ are at the same distance as in $X_\beta$, but  the distance between  points from different spininess  is obtained by the addition of distances of these points from $g$. Since $J (\{X_\beta: \beta < \alpha \})^{(\alpha)} = \{g\},$  this hedgehog space is metric scattered with the one-point  $\alpha$-derivative.

	\begin{pro}\label{poc}
For any ordinal $\alpha$ there exists  a  metric scattered  space  with the one-point $\alpha$-derivative.    \end{pro} 
\begin{proof} If $\alpha\in \omega_1$, then the ordinal $\omega^\alpha +1$  satisfies the thesis. Suppose that for each $\beta < \alpha$ there exists a  metric scattered space $Y_\beta$ such that $Y_\beta^{(\beta)}= \{g_\beta\}$.
If  $\alpha= \beta +1$, then put $$X =  Y_\beta\times (\omega+1)\setminus \{ (y,\omega): y\in Y_\beta \mbox{ and } y\not=g_\beta \}.$$ When $X$ is  equipped with the topology  inherited from the product topology, then $X$ is a metric space such that $X^{(\alpha)}= \{(g_\beta,\omega)\}.$ 
 If $\alpha$ is a limit ordinal, then we construct $X$ adapting the construction of a hedgehog space, compare \cite[4.1.5.]{eng}.  For $\beta < \alpha$, spaces  $Y_\beta$ are homeomorphic to spininess of the hedgehog space $X$ and the point formed by gluing points $g_\beta$, will be the only point in the space $X$ belonging to its $\alpha$-derivative.    \end{proof}

\begin{cor} Let $\frak{m}$ be an infinite cardinal and $\alpha$ be an ordinal such that $\frak{m}\leq \alpha <\frak{m}^+$. There exists a metric scattered space  of the cardinality $\frak{m}$ which has  nonempty $\alpha$-derivative. 
\end{cor}
\begin{proof}  Each metric scattered space $X$, where $X^{(\alpha)}\not=\emptyset$ and  $\frak m \leq \alpha< \frak m^+,$ from  the above proposition  can be constructed so to have the cardinality $\frak m$. \end{proof}

	\section{On proofs of  Mazurkiewicz-Sierpi\'nski  and Knaster-Urbanik theorems} \label{s4}
Let us demonstrate,  how to use the Telg$\acute{\mbox{a}}$rsky idea -- modified here as Theorem \ref{eoc}, to  simplify a proof of the Mazurkiewicz-Sierpi\'nski theorem: \textit{If $X$ is a countable compact metric space, then  $X$ is homeomorphic to the ordinal  $\omega^\alpha n +1$, where $\alpha < \omega_1$ and $n\in \omega$ are uniquely determined}.
Assume that $X$ is a countable compact metric space.  If the derivative  $X^{(1)}$ is empty, then $X$ has to be finite since it is compact, hence   $X$ is homeomorphic to the ordinal $\omega^0 \cdot |X|=1 \cdot |X|$.  If { $|X^{(1)}|= n ,$ where $0 <n \in \omega $}, then $X$ has to be  the sum of $n$ copies of a convergent sequence, hence  
  $X$ is homeomorphic to the ordinal $\omega \cdot n +1$.  
	Assume inductively that if  $N(X) \leq \alpha$, then  $X$ is homeomorphic to the ordinal $\omega^\beta \cdot n +1$, where  $\beta <\alpha$ and $n\in \omega$.  
Now suppose that   $|X^{(\alpha)}|= 1 $. By Theorem \ref{eoc} -  the difference $X\setminus X^{(\alpha)}$ is an infinite sum of pairwise disjoint  closed-open subsets, each one has the empty $\alpha$-derivative.   
The subspace $X \setminus X^{(\alpha)}$ is homeomorphic to the sum $$ (\omega^{\beta_0} \cdot n_0 +1) \oplus (\omega^{\beta_1} \cdot n_1 +1) \oplus \dots ,$$ by the induction conditions. If $\alpha= \gamma +1$, then one can assume that every $\beta_n=\gamma$. If $\alpha$ is a limit ordinal, then every $\beta_n< \alpha$ and $\lim_{n\to \infty} \beta_n = \alpha$. In both cases we obtain  that  $X$ is homeomorphic to $\omega^\alpha  +1$. If { $|X^{(\alpha)}|= n \in \omega $}, then $X$ has a finite open cover $\mathcal U$ such that each $V\in \mathcal U$ meets 
	$X^{(\alpha)}$ at a single point and members of $\mathcal U$ are pairwise disjoint.  Therefore   $X$ is homeomorphic to  the sum of $ (\omega^\alpha +1)$ taken $n$-times and consequently $X$ is homeomorphic to $\omega^\alpha \cdot n +1$.

	Recall that B. Knaster and K. Urbanik \cite{ku} proved  that any  countable metric scattered  space is homeomorphic to a subset of a countable ordinal. Therefore, it has a metric scattered compactification, which is a closed subset of some $\beta +1$, where $\beta < \omega_1$.  A proof that any  countable metric scattered  space has a countable 
	metric compactification, which is  scattered,  was also presented in \cite[p. 25]{kur}. For compact $X$, the proof by S. Mazurkiewicz and W.
	Sierpi\'nski  indicates the smallest ordinal number in which  $X$ can be  embedded. For any countable   metric scattered space a similar indication is not clearly described. So, let us describe the ordinals, which are essential for the   induction proof of the Knaster-Urbanik theorem. When $\alpha$ is a countable  ordinal, let $E(\alpha)$ be the least ordinal such that any countable metric  scattered space with the one-point $\alpha$-derivative  	can be 
embedded  into 	$E(\alpha)$. Thus 	$E(0)=1$ and 	$E(1)=\omega^2+1$, and also $E(m)=\omega^{2m}+1$ for any $m\in \omega$.  In fact, we have the following version of Gillam Lemma 8, see  \cite{gil}.

\begin{pro} \label{p9} If $m\in \omega$, then $E(m)\leq \omega^{2m}+1$.  \end{pro}
\begin{proof} Suppose $X$ is a countable metric space such that  $X^{(1)} =\{ g \}$. Let $\{U_n: n\in \omega\}$ be a decreasing base at the point $g$. Then any one-to-one function $f: X \to \omega^2+1$ such that $f(g)=\omega^2$ and any image $f[U_n\setminus  U_{n+1}]$ is contained in the interval $(\omega \cdot n, \omega\cdot (n+1))$ has to be an embedding of $X$ into $\omega^2+1$. Therefore $E(1)\leq \omega^{2}+1$.

Assume  that if  $Y$ is a countable metric space such that  $Y^{(m-1)} =\{ h \}$,  then  $\omega^{2m-2}+1$ contains a homeomorphic copy of $Y$ such that the point $h$ corresponds to the ordinal $\omega^{2m-2}$. Suppose $X$ is a countable metric space such that  $X^{(m)} =\{ g \}$. 
Choose a family $\{U_n: n\in \omega\}$ of closed-open sets  such that it is a decreasing base at the point $g$ and  each set $U_n\setminus U_{n+1}$ intersects $X^{(m-1)}$. By Theorem \ref{eoc}, each $U_n\setminus U_{n+1}$  is an union of pairwise disjoint closed-open sets $Y_{n,k}$ such that $$Y_{n,k}\cap X^{(m-1)}=\{g_{n,k}\}.$$ Based on inductive assumptions, there exist embeddings $$f_{n,k}: Y_{n,k}\to \omega^{2m-2}+1$$ such that each point $g_{n,k}$ corresponds to 
the ordinal $\omega^{2m-2}$.  Line up images $ f_{n,k} [ Y_{n,k}]$ such that  $ f_{n,i} [ Y_{n,i}]$  followed by  $ f_{n,i+1} [ Y_{n,i+1}]$,  for $i\in \omega$. We get   embeddings $f_n: U_n \setminus U_{n+1} \to \omega^{2m-1}. $   Again, line up images $ f_{n} [ U_n \setminus U_{n+1}]$ and ordinals $\{\omega^{2m-1}\cdot k: 0<k\in \omega \}$ such that 
 $ f_{0} [ U_0 \setminus U_{1}]$ followed by  $\{\omega^{2m-1}\}$ followed by $ f_{1} [  U_1 \setminus U_{2}]$  followed by  $\{\omega^{2m-1}\cdot 2 \}$ followed by $ f_{2} [  U_2 \setminus U_{3}]$ and so on. Except for $n=0,$ we have  $$ f_{n} [ U_n \setminus U_{n+1}]\subset [\omega^{2m-1}\cdot n +1, \omega^{2m-1}\cdot (n+1)] = (\omega^{2m-1}\cdot n , \omega^{2m-1}\cdot (n+1)+1).$$ This means that  images $ f_{n} [ U_n \setminus U_{n+1}]$ are contained in pairwise disjoint closed-open intervals. So, we get  the embedding $f: X  \to \omega^{2m}+1,$ as far as we put $f(g)=\omega^{2m}.$ Therefore $E(m)\leq \omega^{2m}+1$.
\end{proof}

\begin{cor} \label{c9} If $m\in \omega$, then $E(m)= \omega^{2m}+1$.  \end{cor}
\begin{proof} Let $X(1)=\omega^2+1 \setminus \{\omega\cdot k: k\in \omega \}.$ So $X(1)^{(1)} = \{\omega^2\}$. Without loss of generality, we can assume that  $f:X(1) \to \omega^2+1$  is an embedding such that  $$\beta =f(\omega^2) = \sup f[X(1)].$$ Put $\beta_1 = \sup f[ \omega]$ and  $f[X(1)]\cap [0, \beta_1]= A_1$.  Since $A_1$ is infinite, we have  $\beta > \beta_1 \geq \omega$. 
Inductively assume that ordinals $\beta_1, \beta_2, \dots, \beta_{n-1}$ and discrete infinite subspaces   $A_1, A_2, \dots, A_{n-1}\subset f[X(1)]$ are already defined and $\beta_k=\sup A_k \geq \omega \cdot k$, for  $0<k<n$. Choose an infinite and discrete subspace $$A_n \subseteq f[X(1)] \cap (\beta_{n-1}, \beta)$$
 and put $ \beta_n= \sup A_n $.   Assuming inductively that $ \beta_{n-1} \geq \omega \cdot (n-1)$ we get $ \beta_n \geq \omega \cdot n$. This implies  $\omega^2\leq \lim_{n\to \infty} \beta_n \leq \beta $. Therefore $E(1) = \omega^2+1$. 

Let $m>1$. Assume that  the space $X(m-1)\subseteq \omega^{2m-2}+1$ is already defined such that   $X(m-1)^{(m-1)} = \{\omega^{2m-2}\}$ and $X(m-1)$ can not be embedded 
into $\beta < \omega^{2m-2}.$ Take a countable infinite family $\mathcal S$  consisting of copies  $ X(m-1)$. Let  $X(m)= \bigcup \mathcal S \cup \{g\}$ be equipped with the topology, where $\bigcup \mathcal S$ inherits the sum topology and the point $g$ has a decreasing base of neighborhoods $\{U_n: n \in \omega\}$ such that each $U_n \setminus U_{n+1}$ is the union of an infinite many copies of  $X(m-1)$. By the definition, $X(m)$ can be embedded into  $\omega^{2m}+1$ such that the point $g$ corresponds to the ordinal $\omega^{2m}.$
Without loss of generality, we can assume that  $f:X(m) \to \omega^{2m}+1$  is an embedding such that  $\beta =f(g) = \sup f[X(m)].$ Put $\beta_0 = \sup f[U_0\setminus U_1]$ and  
$f[X(m)]\cap [0, \beta_0]= A_0$.  By the induction assumptions,  we get  $\beta > \beta_0 \geq \omega^{2m-1}$. 
Inductively assume that ordinals $\beta_0, \beta_1, \dots, \beta_{n-1}$ and  subspaces   $A_0, A_1, \dots, A_{n-1}\subset f[X(m)]$ are already defined such that 
$$\beta > \beta_k=\sup A_k \geq \omega^{2m-1} \cdot (k+1),$$ for each $k<n$. Let $A_n \subset f[X(m)] \cap (\beta_{n-1}, \beta)$ be an infinite union of copies of $X(m-1)$ such that $\beta > \sup A_n=\beta_n$. Since $ \beta_{n-1} \geq \omega^{2m-1} \cdot n$ we get  $ \beta_n \geq \omega^{2m-1} \cdot (n+1)$. This implies  $\omega^{2m}\leq \lim_{n\to \infty} \beta_n \leq \beta $. Therefore $E(m) = \omega^{2m}+1$.
\end{proof} 

Defined in the above proof spaces $X(m)$ can be added the same way as ordinals, except that the result of such addition must be equipped with the  inherited topology. 
However, such an extension rules seem to be a good topic for future research. 
\begin{pro}\label{t2} Let $\alpha = \gamma +m$, where $m\in \omega$ and $\gamma < \omega_1$ is a limit ordinal. Then  $E(\alpha)=\omega^{\gamma + 2m+1}+1$.
 \end{pro}
\begin{proof} Let the space $X(\omega)$ be such that   $X(\omega)^{(\omega)}=\{g\}. $ Moreover, the point $g$ has a decreasing base of neighborhoods $\{U_n: n \in \omega\}$ such that each $U_n \setminus U_{n+1}$ is an infinite sum of copies of $X(k)$, defined in the proof of Corollary \ref{c9},   where $k $ runs by infinitely many natural numbers. 
Without loss of generality, we can assume that  $f:X(\omega) \to \omega^{\omega +1}+1$  is an embedding such that  $\beta =f(g) = \sup f[X(\omega)].$ Put $\beta_0 = \sup f[U_0\setminus U_1]$ and  $f[X(\omega)]\cap [0, \beta_0]= A_0$.  By the induction assumptions,  we have  $\beta > \beta_0 \geq \omega^{\omega}$. 
Inductively assume that ordinals $\beta_0, \beta_1, \dots, \beta_{n-1}$ and  subspaces   $A_0, A_1, \dots, A_{n-1}\subset f[X(\omega)]$ are already defined and $$\beta > \beta_{n-1}=\sup A_{n-1} \geq \omega^{\omega} \cdot n.$$ Let $A_n \subset f[X(\omega)] \cap (\beta_{n-1}, \beta)$ be an infinite sum of copies of $X(k),$ where $k$ runs by infinitely many natural numbers. We get  $$\beta > \sup A_n=\beta_n> \beta_{n-1} \mbox{ and } \, \beta_n \geq \omega^{\omega} \cdot (n+1).$$ Assuming inductively that $ \beta_{n-1} \geq \omega^{\omega} \cdot (n-1)$ we get  $ \beta_n \geq \omega^{\omega} \cdot n$.  Therefore $\omega^{\omega+1}\leq \lim_{n\to \infty} \beta_n \leq \beta $ and $E(\omega) = \omega^{\omega +1}+1$.
Similarly, one can prove that $E(\gamma) = \omega^{\gamma +1}+1$ for each limit ordinal $\gamma < \omega_1$. And also in analogy to the proof of Corollary \ref{c9}, one can  get 
$E(\gamma+ m ) = \omega^{\gamma+2m +1}+1$, whenever $m\in \omega$ and $\gamma < \omega_1$ is a limit ordinal.
 \end{proof} 

\begin{pro}\label{p8} If $0<\alpha < \omega_1$, then any  countable metric  space with nonempty $\alpha$-derivative contains a homeomorphic copy of $\omega^\alpha +1$.
			\end{pro}
			\begin{proof} Let $X$ be a countable metric  space. Without   loss of generality, assume that  $X^{(\alpha)}=\{g\}.$
			If $\alpha =1$, then $X$ contains a convergent sequence, which is homeomorphic to   $\omega +1$. Suppose, that the thesis holds for all $\beta < \alpha$. Fix a metric $\varrho$ on $X$. Choose nonempty closed-open sets $V_n \subseteq X\setminus \{g\}$ such that $V_n \subseteq B(g, \frac{1}{n})$. By the induction assumptions each $V_n$ contains a homeomorphic copy of $\omega^{\beta_n}+1$, where $\beta_n < \alpha$. So, we choose  copies of $\omega^{\beta_n}+1 \subseteq V_n$ such that $\omega^\alpha = \sum_{n\in \omega}\omega^{\beta_n}.$ The sum of these copies plus point $g$ gives a subspace homeomorphic to $ \omega^\alpha +1.$    			
	\end{proof}
	
	\section{More on local embeddable properties} \label{s5}
	
	Let $\mathcal A$ be the poset consisting of dimensional types of countable metric spaces $X$ with  $1 < N(X) \in \omega$. Many properties of $ ({\mathcal P}(\Bbb Q)/\!=_E ,\leq_d)$ can be reduced to $\mathcal A$,   as it is observed in \cite[p. 69 - 81]{gil}. 
Let us discuss another local embeddable  invariants, which are not mentioned in the paper \cite{gil}. Assume that $X$ is a metric scattered space such that $X^{(m)}=\{g\}$, where $0<m\in \omega$.  We  say that $X$ has $(m,1)$-\textit{stable dimensional type} if no $Y\subseteq X$  has  smaller dimensional type than $X$, whenever  $X \setminus Y$ is a closed-open set and $g\in Y$. There exist exactly two $(1,1)$-stable dimensional types, i.e. the dimensional type of the convergent sequence $G= \omega +1$ or  the dimensional type of the subspace $I= \omega^2+1 \setminus \{\omega, \omega \cdot 2, \omega \cdot  3, \ldots\}$. So, $I$  is a space with the single cluster point which has a base of open neighborhoods $\{U_n: n\in \omega\}$ such that each difference   $U_n \setminus U_{n+1}$ is  infinite and discrete. 

 We  leave the readers check that there exist exactly five $(2,1)$-stable dimensional types. 
These are dimension types of  following spaces: 
\begin{itemize}
	\item[] $\omega^2+1$; \item[]  $\omega^3+1 \setminus \{\omega^2, \omega^2 \cdot 2, \omega^2 \cdot  3, \ldots\}$; \item[] $\sum_\omega I +1\subset \omega^3+1$, where the subspace is  established as a  sequence of $I$ followed by a copy of $I$ (infinitely many times) and with $1$ at the end;  \item[] $\sum_\omega I \oplus\sum_\omega I\oplus \sum_\omega I\oplus \ldots +1\subset \omega^4+1$, where the subspace is established as a  sequence of $\sum_\omega I$ followed by a copy of $\sum_\omega I$ (infinitely many times) with $1$ at the end and with the ordinals $\omega^3, \omega^3 \cdot 2, \omega^3 \cdot 3, \ldots $ thrown out;    \item[] $\sum_\omega (\omega^2 \oplus I ) +1 \subset \omega^3+1$, where  operation  $\sum_\omega (...) +1 $ is used  as above and $\omega^2 \oplus I\subset \omega^2 \cdot 2+1$  is the subspace of established as a copy of $\omega^2$  followed by a copy of $I$ with $\omega^2$  thrown out.   
\end{itemize}

If   $0<n\in \omega$ and $X \in \mathcal A,$ then we can prove the following. 
\begin{thm}   There exist finitely many $(n,1)$-stable dimensional types. Each  $X\in \mathcal A$ is a sum of closed-open subspaces  with  $(k,1)$-stable dimensional types, where $0<k<N(X)$.     \end{thm} 
\begin{proof} For $n=1$ and $n=2$ the
theses are fulfilled. Let $S_{n-1}$ be the family of all $(k,1)$-stable dimensional types, where $k<n$. For inductive proof, assume that $S_{n-1}$ is finite and each space $Y\in \mathcal A$, such that $N(Y) \leq n$,  is a sum of  closed-open subspaces  with  $(k,1)$-stable dimensional types.
Consider a space $X$ with  the $(n,1)$-stable dimensional type such that $X^{(n)}=\{ g \}$.  By  Theorem \ref{eoc}, the subspace $X\setminus \{g\}$ can be divided into pairwise closed-open sets  with Cantor-Bendixon rank equal to $n$. Therefore and by the induction assumptions, the subspace $X\setminus \{g\}$ can be divided into finitely many closed-open sets,  each of which consists of pairwise disjoint closed-open sets with the same $(k,1)$-stable dimensional type,  belonging to   $ S_{n-1}$. Denote $\mathcal V$ the family of all relevant dimensional types  for $X\setminus \{g\}$.   
Fix a decreasing base $\{ U_n: n \in \omega\}$ of open neighborhoods of the point $g$ such that each $U_n \setminus U_{n+1}$ contains a single closed-open set which dimensional type is from    $\mathcal V$ or infinitely many such sets. Since $X$ has the $(n,1)$-stable dimensional type, therefore  the dimensional type of $X$ depends only on whether any dimensional type of $\mathcal V$  occurs in $U_n \setminus U_{n+1}$ at most once or at least infinitely many times.  Such opportunities are finitely many. \end{proof}

We do not know whether the cardinality of families  $S_n$ may  well be bounded by a polynomial in $n$. However, the concept of $(k,1)$-stable dimensional types makes it easier to understand the results on poset $(\mathcal A, <_E)$ and simplifies some of the reasoning from the paper \cite{gil}. In our opinion, combinatorial properties of families  $S_n$ require further examination, but that is a topic for future research.

		\section{Dimensional types of uncountable subspaces of $\omega_1$}
Let $\mathbb Y$ be the sum of all countable ordinals. Thus,  $\omega_1$ contains a  homeomorphic copy of $\mathbb Y$. Hence,  $\mathbb Y$ is a metric space which has a smaller dimensional type than  the not metric space  $\omega_1$. 
  The space  $\mathbb Y$ is special among the uncountable subspaces 
of $\omega_1$.    Namely, if   $X\subset \omega_1$ is a metric subspace, then $X<_E \mathbb Y$. Indeed, take a  open cover $\mathcal U$ of $X$, which consists of countable sets.  Then, use any disjoint open refinement of $\mathcal U$ to construct a required embedding.

			\begin{pro}\label{t7}  If a subspace $X\subseteq \omega_1$  contains a homeomorphic copy of any countable ordinal, then  $\mathbb Y <_E X.$ 
			\end{pro}
			
			\begin{proof} Let  $\mathbb F=\{\Iee_\alpha: \alpha < \omega_1 \}$ be a family of closed  and pairwise disjoint  intervals  of $\omega_1$  such that each intersection $X \cap \Iee_\alpha$ contains a homeomorphic copy of $\omega^\alpha +1$. Then  $X \cap \bigcup\mathbb F$ contains a copy of  $\mathbb Y$.
	\end{proof}
	
	Recall that a set $S \subseteq \omega_1$  is   \textit{stationary}, if $S$ intersects  any  closed and unbounded subset of $\omega_1$, compare \cite[p. 78]{kun}. Well-known Solovay's result  says that each stationary set can be divided into uncountably many stationary sets, compare \cite{jec}. 
	Note that, if $X\subseteq \omega_1$ is not stationary, then $X$ is a metric $\sigma$-discrete  space. Indeed, any complement of a closed unbounded set is an union of pairwise disjoint open intervals of ordinals. Each such interval has to be countable. Therefore $X$ is contained in a sum of metric spaces. By Theorem \ref{ms}, it  has to be $\sigma$-discrete. 
	
	\begin{pro}\label{l3} If   $X\subset \omega_1$ is a discrete subspace, then $X$ is not stationary. \end{pro}
	\begin{proof} If $X$ is bounded by an ordinal $\alpha < \omega_1$, then $X$ is disjoint to the closed and unbounded interval $(\alpha, \omega_1)$, so we can assume that $X$ is unbounded in $\omega_1$.  Let $\{(a_\alpha, b_\alpha): \alpha \in X \} $ be an uncountable  family of pairwise disjoint intervals such that $X \cap (a_\alpha, b_\alpha) = \{\alpha \}$, for each $\alpha \in X$.  Without loss of generality, we can assume that $ \alpha<\beta $ implies $ a_\alpha < b_\alpha \leq a_\beta < b_\beta.$  We get  that  the complement of an open set $$ \bigcup \{(a_\alpha, b_\alpha): \alpha \in X\} \supset X$$ is unbounded, because it contains  $\{b_\alpha: \alpha \in X\}$. \end{proof}

\begin{pro}\label{t8} If $X$ is a stationary set and $\alpha < \omega_1$, then $X^{(\alpha)}\not=\emptyset$. 
			\end{pro}
			\begin{proof} Suppose $X^{(\alpha)} =\emptyset$, where  $\alpha < \omega_1$. Then $$ X= \bigcup \{X^{(\beta)}\setminus X^{(\beta +1)}: \beta < \alpha \} $$ is an union of countably many subspaces with discrete subspace topologies.      Since Proposition \ref{l3} and \cite[p. 78]{kun}, the set  $X$  can not be  stationary. 
				\end{proof}

Following M. Ismail and A. Szyma\'nski \cite{is},   the \textit{discrete metrizability number} of a space $X$, denoted $dm(X),$ is the smallest cardinal number $\kappa$ such that $X$ can be represented as a union of $\kappa$ many discrete subspaces. But  the \textit{ metrizability number}  $m(X),$ is the smallest cardinal number $\kappa$ such that $X$ can be represented as a union of $\kappa$ many metric subspaces. We have the following. 
	\begin{cor}\label{c7} If  $X\subseteq \omega_1$ is a stationary set, then $dm(X) = \omega_1=m(X)$.\end{cor}
	\begin{proof} A stationary set can not be a union of countably many not stationary subsets. Hence, we get  $dm(X) = \omega_1=m(X)$, using  Theorem \ref{ms} and Proposition \ref{l3}. \end{proof}
	
	\begin{lem}\label{sta} If $X\subseteq \omega_1$ and  $f: X \to \omega_1$ is an embedding, then there exists a closed unbounded set $C$ such that $$ f[X] \cap C =X \cap C$$  \end{lem}
	\begin{proof} For  countable $X$, the set $\{ \alpha: \sup\{X \cup f[X]\} < \alpha \}$ is what we need. Suppose $X$ is uncountable. Successively by induction choose strictly greater ordinals $x_{n, \alpha}\in X$ and $y_{n, \alpha}\in f[X]$ such that $$ x_{\beta, k} < x_{\alpha,n} < y_{\alpha,n} < x_{\alpha, n+1},  $$ where $k,n \in \omega$ and $\beta < \alpha < \omega_1$. These ordinals constitute an increasing sequence  lexicographical ordered with respect to indices.  Cluster  points of the set  of these ordinals, give the required closed unbounded set. 
	\end{proof}
	
	Obviously, the above lemma follows that disjoint stationary sets have not comparable dimensional types.
	
	\begin{thm} If $X $ is a stationary set, then the poset \mbox{$({\mathcal P}(X)/\!\!\!=_E,\leq_d)$ } contains uncountable  anti-chains and uncountable strictly descending chains. \end{thm} \begin{proof}  Let $\{S_\alpha: \alpha < \omega_1\}\subseteq X$ be a family of pairwise disjoint stationary set, a such family exists since the mentioned above result of R. Solovay. Since Lemma \ref{sta}, elements of this family have not comparable dimensional types. Also for the same reasons, sets $X_\beta = \bigcup \{S_\alpha : \beta < \alpha \}$ constitute an uncountable strictly descending chains, with respect to the order $<_E$. 	
	\end{proof}
	
\section{Generalized  Knaster-Urbanik Theorem}

Generalizing the above proof of Knaster-Urbanik Theorem, and using Theorem \ref{eoc},  we get a proof of the following result by R. Telg$\acute{\mbox{a}}$rsky \cite{tel}, compare \cite{af}.

\textbf{Corollary.} \textit{Any  metric scattered  space is homeomorphic to a subset of an ordinal number.}

\begin{proof} If $X$ is a discrete space, then $X$ can be embedded  into a set of non-limit ordinals, which has to be a subset of some ordinal.  Suppose $X$ is a metric space such that $X^{(\alpha)}=\{g\}$, where $\alpha >0$. Assume that any subspace $Y\subseteq X$ can be embedded into the ordinal $E(Y)$,  as long as $Y^{(\beta)}$ has exactly one point and  $\beta <\alpha$. Without loss of generality, we can assume that $f_Y: Y\to E(Y)$ is an embedding such that $$f_Y[Y^{(\beta)}] =\{\sup E(Y)\}, \mbox{ where } |Y^{(\beta)}|=1 .$$  Let $\{U_n: n\in \omega \}$ be a decreasing base of neighborhoods of $g$ consisting of closed-open sets.  By Theorem \ref{eoc}, there exist pairwise disjoint closed-open sets $Y_{\xi, n} \subseteq U_n \setminus U_{n+1}$ such that for each $Y_{\xi, n}$ has exactly one point derivative $Y_{\xi, n}^{(\beta)}$, where $\beta < \alpha$.   We order ordinals $E( Y_{\xi
, n})$  as follows:  $E( Y_{\xi
, n}) $ followed  by $E( Y_{\nu
, n})$, with respect to the order of first indexes, and with $1$ at the end.   In the next step,  we order similarly ordinals  $E( Y_{\xi
, n+1})$ and place them, keeping their order, after $1$ located at the end of an ordered collection in the previous step. Finally we put the point $g$.  The union of all   $f_{Y_{\xi,m}} $ contained in the corresponding $E(Y_{\xi,m})$, which are ordered as above,   gives the required embedding. 
\end{proof}
If  $\frak m $ is an infinite cardinal number, then $\frak m^+$ denotes the least cardinal number greater than $\frak m $.
Thus,  the above corollary can be formulated more precisely.

\begin{pro} \label{p22}
 Any  metric scattered  space of the cardinality $\frak m$ is homeomorphic to a subset of  an ordinal $\alpha < \frak m^+$.  \end{pro}
\begin{proof} If a  metric  scattered space $X$ has the cardinality $\frak m$ and  $X^{(\alpha)}$ is the last non-empty derivative, then $\alpha < \frak m^+$. It is enough to see that with the same proof as for the above corollary, the space $X$ is embeddable in $\frak m^+$.
\end{proof}

\section{Non-homeomorphic metric scattered spaces}

Let us start with an improvement of   Mazurkiewicz-Sierpi\'nski Theorem \cite[Th\'eor$\grave{\mbox{e}}$me 3]{ms}, which says that there is continuum many non-homeomorphic countable metric scattered spaces.

\begin{pro}\label{p88} The ordinal $\omega^\omega$ contains  continuum many  non-homeomorphic subspaces.
			\end{pro} 
			\begin{proof} For a  binary sequence $(f_1, f_2, \dots  )$ define inductively  scattered spaces $X(f_1, f_2, \dots , f_m)$, with  the one-point $m$-derivative   $\{h_m\}$. Put $X(0)= G$ and  $X(1)=I$, where spaces $G$ and $I$  are the same as it is defined in Section \ref{s5}. The cluster points of $G$ and $I$ can be denoted $g_G$ and $g_I$, respectively.  If a space $X(f_1, f_2, \dots , f_n)$ is already defined, then let $$ X(f_1, f_2, \dots , f_n, 0)= X(f_1, f_2, \dots , f_n) \times (G\setminus \{g_G\}) \cup \{(h_n, g_G)\}$$  be a subspace of the product space $X(f_1, f_2, \dots , f_n) \times G.$  And let 
$$ X(f_1, f_2, \dots , f_n, 1)= X(f_1, f_2, \dots , f_n) \times (I\setminus \{g_I\}) \cup \{(h_n, g_I)\}$$			
			   be the subspace of the product space $X(f_1, f_2, \dots , f_n) \times I.$
				
			If $f=(f_1, f_2, \ldots )$ is an infinite binary sequence, then let $X_f$ be the sum of spaces $\{X(f_1, f_2, \ldots, f_n): 0< n \in \omega\}$. So,  we have $X_f^{(\omega)}=\emptyset$. Also, if $0<n$ and $f_n=0$, then the difference $X_f^{(n-1)}\setminus X_f^{(n+1)}$ is a subspace which consists of pairwise disjoint  closed-open (with respect to the inherited topology) sets homeomorphic to a convergent sequence. But if $f_n=1$, then  the difference $X_f^{(n-1)}\setminus X_f^{(n+1)}$ has no closed-open subset  which is homeomorphic to a convergent sequence. Therefore $\{ X_f: f \in 2^\omega\}$ is a family of non-homeomorphic subspaces of the ordinal  $\omega^\omega$, what we need.\end{proof}

			Consider  the sum  of  $\omega$ many copies of a space $X$.  We  defined the spaces $G(X)$ and $I(X)$ by adding a new point $g$, with a countable base of neighborhoods, to this sum. Points belonging to the sum have  unchanged bases of neighborhoods. The point  $g$ has a decreasing base $\{U_n: n \in \omega\}$ such that $U_n \setminus U_{n+1}$ consists of copies of $X$   as closed-open subsets. So,  in $G(X)$  each  $U_n \setminus U_{n+1}$ consists of a single copy of $X$. However,  each  $U_n \setminus U_{n+1}$ consists of infinitely many  copies of $X$ in $I(X)$.			In particular,  $G=G(1)$ and $I=I(1)$.
			
\begin{thm}\label{t88} For each infinite cardinal number $\frak m$, there exist $2^{\frak m}$ many  non-homeomorphic metric spaces of the cardinality $\frak m$, each one with empty $\frak m$-derivative.
			\end{thm} 
			\begin{proof} Since Proposition \ref{p88}, we can assume that $\frak m$ is an uncountable cardinal.
			For every  binary sequence $f=\{f_\beta: 0<\beta < \frak m\}$  define inductively a scattered space $Y(f_1, f_2, \ldots, f_{\beta })$ as follows. Put $Y(0)= G$  and $Y(1)=I$. 
			Suppose that   metric  scattered spaces $Y(f_1, f_2, \ldots, f_{\delta})$ are already defined,  for  $\delta <\beta$.  If $\beta $ is a limit ordinal, then put $$ Y(f_1, f_2, \ldots, f_{\beta}) = J(\{Y(f_1, f_2, \ldots, f_{\delta}): \delta < \beta  \}). $$			
	If $\beta $ is a non-limit ordinal, then put $$Y(f_1, f_2, \ldots, f_{\beta -1}, 0)=G(Y(f_1, f_2, \ldots, f_{\beta -1}))$$ and $$Y(f_1, f_2, \ldots, f_{\beta -1}, 1)=I(Y(f_1, f_2, \ldots, f_{\beta -1})).$$ 
		Finally, let $Y(f)$ be the sum 
 of spaces $Y(f_1, f_2, \ldots, f_{\beta })$, where $\beta < \frak m$.

By the definition, if $\beta < \frak m$, then each space $Y(f_1, f_2, \ldots, f_{\beta})$ has the cardinality less than $\frak m$. We also have 
$Y(f_1, f_2, \ldots, f_{\beta})^{(\frak m)}= \emptyset,$ hence $Y(f)^{(\frak m)}= \emptyset.$ Bearing above in mind and using  Proposition \ref{p22}, one can check that each $Y(f)$ embedds into $\frak m.$ Since each $Y(f)$ has the cardinality $\frak m$, it remains to show that the family $\{ Y(f): f \in 2^\frak m\}$ contains a subfamily of cardinality $2^\frak m$ consisting of non-homeomorphic metric scattered space. Indeed, if $\gamma < \frak m$ is a non-limit ordinal and $f(\gamma)\not= g(\gamma) $, where $f,g \in 2^\frak m$, then the subspaces $Y(f)^{(\gamma)}\setminus Y(f)^{(\gamma+2)}$ and $Y(g)^{(\gamma)}\setminus Y(g)^{(\gamma+2)}$ are not homeomorphic, since one of them consists of closed-open subsets  homeomorphic to $I$, but the second contains  no homeomorphic copy of $I$. 
			\end{proof}


\begin{thebibliography}{40}
\bibitem{af} A. Arosio, A. V. Ferreira, 
\textit{On nonseparable 0-dimensional metric spaces}. 
Portugal. Math.  37  (1978), no. 3-4, 273 - 297 (1981).
\bibitem{eng} R. Engelking, \textit{General Topology}, 
 Polish Scientific Publishers, Warsaw, (1977).
\bibitem{enw} R. Engelking,  \textit{Teoria wymiaru}. (Pañstwowe Wydawnictwo Naukowe, Warsaw, (1977).
\bibitem{fre} M. Fr$\acute{\mbox{e}}$chet, \textit{Les dimensions d'un ensemble abstrait}, 
Math. Ann., 68 (1910), 145 - 168.
 \bibitem{gil}  W. D. Gillam, \textit{Embeddable properties of countable metric spaces}. Topology Appl.  148  (2005),  no. 1-3, 63 - 82.
\bibitem{is} M. Ismail, A. Szyma\'nski,  \textit{On the metrizability number and related invariants of spaces}. Topology Appl.  63  (1995),  no. 1, 69 - 77.
\bibitem{jec} T. Jech, \textit{Set theory}. Perspectives in Mathematical Logic. Springer-Verlag, Berlin, (1997).
\bibitem{kr} V. Kannan, M.  Rajagopalan, \textit{On scattered spaces}.
Proc. Amer. Math. Soc. 43 (1974), 402 - 408.
\bibitem{ku} B. Knaster, K. Urbanik, \textit{Sur les espaces complets séparables de dimension 0}. 
 Fund. Math.  40,  (1953). 194 - 202.
\bibitem{kun} K. Kunen, \textit{Set theory. An introduction to independence proofs.} Studies in Logic and the Foundations of Mathematics, 102. North-Holland Publishing Co., Amsterdam-New York, (1980).
\bibitem{kur1} K. Kuratowski, \textit{Topology-Volume I}.  Transl. by J. Jaworowski, Academic Press, New York-London; Pañstwowe Wydawnictwo Naukowe Polish Scientific Publishers, Warsaw (1966). 
\bibitem{kur} K. Kuratowski, \textit{Topology-Volume II}.  Transl. by A. Kirkor Academic Press, New York-London; Pañstwowe Wydawnictwo Naukowe Polish Scientific Publishers, Warsaw (1968). 
\bibitem{ms} S. Mazurkiewicz, W. Sierpi\'nski,   \textit{Contribution $\grave{\mbox{a}}$ la topologie des ensembles 
d$\acute{\mbox{e}}$nombrables }.  Fund. Math. I (1920), 17 - 27.
\bibitem{ple}  Sz. Plewik, \textit{ On subspaces of the Pixley-Roy example}. 
Colloq. Math.  44  (1981), no. 1, 41 - 46.
\bibitem{prz} T. C. Przymusi\'nski, 
\textit{Normality and paracompactness of Pixley-Roy hyperspaces}. 
Fund. Math.  113  (1981), no. 3, 201 - 219.
\bibitem{sie1} W. Sierpi\'nski, 
\textit{Introduction to General Topology}. Lecturer in Mathematics at the University of Toronto.   The University of Toronto Press (1934). 
\bibitem{sie} W. Sierpi\'nski, 
\textit{General topology}.
Mathematical Expositions, No. 7,
University of Toronto Press, Toronto (1952).
\bibitem{tel} R. Telg$\acute{\mbox{a}}$rsky,
\textit{Total paracompactness and paracompact dispersed spaces}. 
Bull. Acad. Polon. Sci.  Math. Astronom. Phys. 16 (1968), 567 - 572.
\end{thebibliography}
\end{document}